\documentclass{amsart}
\usepackage{amsfonts,amssymb,amsmath,amsthm}
\usepackage{url}
\usepackage{enumerate}

\newtheorem{thrm}{Theorem}[section]
\newtheorem{lem}[thrm]{Lemma}

\theoremstyle{definition}

\newtheorem{remark}[thrm]{Remark}

\numberwithin{equation}{section}

\def\sqr#1#2{{\,\vcenter{\vbox{\hrule height.#2pt\hbox{\vrule width.#2pt
height#1pt \kern#1pt\vrule width.#2pt}\hrule height.#2pt}}\,}}
\def\bo{\sqr44\,}

\newcommand{\vp}{\varphi}

\author{Cho-Ho Chu and Lei Li}
\address{
School of Mathematical Sciences,
Queen Mary, University of London,
London E1 4NS, UK} \email{c.chu@qmul.ac.uk}
\address{School of Mathematical Sciences and LPMC,
  Nankai University, Tianjin 300071,
 China}
\email{leilee@nankai.edu.cn}

\thanks{The second author is supported by NSF (China) research grants No.11301285}

\keywords{Separably injective Banach space.  $C_\sigma$-space. Substonean space.
} \subjclass{Primary 46B20, 17C65, 54G05}

\begin{document}

\title[Separably injective $C_\sigma$-spaces]
 {Separably injective $C_\sigma$-spaces}
\begin{abstract}
We show  that a (complex) $C_\sigma$-space is separably
injective if and only if it is linearly isometric to the Banach
space $C_0(\Omega)$ of complex continuous functions vanishing at
infinity on a substonean locally compact Hausdorff space $\Omega$.
\end{abstract}

\maketitle

\section{Introduction}\label{sect1}

Recently, separably injective Banach spaces have been studied in depth by  Avil\'es,  Cabello S\'anchez,  Castillo, Gonz\'alez
 and Moreno in \cite{AC13,av,AC16}, where one can find a number of interesting examples of these spaces despite
  the scarcity of examples of injective Banach spaces.
In contrast to the fact that $1$-injective Banach spaces
 are isometric to the Banach space $C(\Omega)$ of continuous functions on a compact
Hausdorff space $\Omega$ \cite{H58,kelley}, $1$-separably injective Banach spaces
need not be complemented in $C(\Omega)$. In view of this, a  natural question arises: when is a $1$-complemented subspace of $C(\Omega)$ separably injective? We address this question in this paper.

A Banach space $V$ is {\it $1$-separably injective} if every continuous
linear map\\  $T: Y \longrightarrow V$ on a closed subspace $Y$ of  a {\it separable} Banach space $Z$
admits a norm preserving
extension to $Z$.
It is known that $C(\Omega)$ itself is separably injective if and only
if $\Omega$ is an $F$-space \cite{AC13}. Abelian C*-algebras with identity are of the form $C(\Omega)$ for some $\Omega$.
The ones without identity can be represented as the algebra $C_0(S)$ of complex continuous functions vanishing at infinity on a locally compact Hausdorff space $S$ and  it has  been shown lately  that $C_0(S)$ is separably injective if and only if $S$ is substonean \cite{CL16}. Following
\cite{GP}, we call $S$   \textit{substonean} if any two
disjoint open $\sigma$-compact subsets of $S$ have disjoint
compact closures. The compact substonean spaces are exactly the \textit{F-spaces}
defined in \cite{GJ,s}. However, infinite discrete spaces are
F-spaces without being substonean. We refer to \cite[Example
5]{HW89} for an example of a substonean space which is not an
F-space.

Noting that the class of $1$-complemented subspaces of $C(\Omega)$ is identical to that of $C_\sigma$-spaces
\cite[Theorem 3]{lw} (see also Remark \ref{cs}), our question amounts to asking for a characterisation of separably injective $C_\sigma$-spaces.
We give a complete answer by showing that a $C_\sigma$-space is separably
injective if and only if it is linearly isometric to the function
space $C_0(S)$  on a substonean locally compact Hausdorff space $S$.

In what follows, all Banach spaces are over the complex field and we will  denote by $C_0(K)$ the C*-algebra of complex continuous functions vanishing at
infinity on a locally compact Hausdorff space $K$. If $K$ is compact, we omit the subscript $0$.
Given a function $g \in C_0(K)$, we denote by $\overline g$  the complex conjugate of $g$.

Let $\mathbb{T}=\{\alpha\in \mathbb{C}: |\alpha|=1\}$ be the
circle group. By a $\mathbb{T}$-space, we mean a locally compact
Hausdorff space $K$ equipped with a continuous group action
$$\sigma: (\alpha, \omega) \in \mathbb{T} \times K  \mapsto \alpha \cdot \omega
\in K.$$
A complex Banach space is called a {\it complex
$C_\sigma$-space} if it is linearly isometric to a function space
of the form $C_\sigma(K)$ for some $\mathbb{T}$-space $K$, defined
by
$$C_\sigma(K) = \{f\in C_0(K): f(\sigma(\alpha,\omega))= \alpha f(\omega),
\forall \omega \in K\}.$$
We note that the definition of a complex $C_\sigma$-space in  \cite{olsen} requires
$K$ to be compact.

To achieve our result, we make substantial use of the Jordan algebraic structure
of $C_\sigma(K)$. Indeed, although $C_\sigma(K)$ lacks a C*-algebraic structure, it is equipped
with a triple product
$$\{f,g,h\} = f\overline g h \qquad (f,g, h \in C_\sigma(K))$$
which turns it into a {\it JB*-triple} with many useful Jordan properties.

First, let us give a brief introduction to JB*-triples which generalise  C*-algebras.
For further references and the geometric origin of JB*-triples, we refer to \cite{book,ru,u}. A complex Banach space $V$ is a {\it JB*-triple} if it admits a continuous triple product
$$\{\cdot,\cdot,\cdot\} : V^3 \longrightarrow V$$
which is symmetric and linear in the
outer variables, but conjugate linear in the middle variable, and
satisfies
\begin{enumerate}
\item[(i)] $\{x,y,\{a,b,c\}\}=\{\{x,y,a\},b,c\}-\{ a,\{y,x,b\},c\}
+\{a,b,\{x,y,c\}\}$; \item[(ii)]  the operator $a\bo a : x\in V \mapsto \{a,a, x\} \in V$ has real numerical range and
non-negative spectrum; \item[(iv)] $\|a\bo a\|=\|a\|^2$
\end{enumerate}
for $a,b,c,x,y \in V$. We always have $\|\{a,b,c\}\| \leq \|a\| \|b\| \|c\|$ and (i)
is called the {\it Jordan triple identity}.
A C*-algebra $A$ is a JB*-triple with the triple product
$$\{a,b,c\} = \frac{1}{2}(ab^*c +cb^*a) \qquad (a,b,c \in A).$$
More generally, the range of a contractive projection on a C*-algebra is a JB*-triple
(cf.\,\cite[Theorem 3.3.1]{book}), but not always a C*-algebra.
An element $e$ in a JB*-triple is called a {\it tripotent} if $e = \{e,e,e\}$. Tripotents in
C*-algebras are exactly the partial isometries.

A subspace $W$ of a JB*-triple $V$ is called a {\it
subtriple} if $a,b,c \in W$ implies $\{a,b,c\} \in W$. Closed subtriples of a JB*-triple
are JB*-triples in the inherited norm and triple product. A {\it
triple ideal} of $V$ is a subspace $J\subset V$ such that
$\{a,b,c\}\in J$ whenever one of $a$, $b$ and $c$ belongs to $J$.
Given a closed triple ideal $J\subset V$, the quotient space $V/J$ is a JB*-triple
in the triple product
$$\{a+J, b+J, c+J\} := \{a,b,c\} + J.$$ Two elements $a,b \in V$ are said to be {\it orthogonal}
to each other if $a\bo b =0$, where $a\bo b$ is the continuous linear
 map $x\in V \mapsto \{a,b,x\}\in V$. Two subspaces $I, J \subset V$ are {\it orthogonal} if
$I\bo J := \{a\bo b: a\in I, b\in J\} = \{0\}$.

The bidual $V^{**}$ of
a JB*-triple $V$ carries a natural structure of a JB*-triple, with a unique predual, in which
the triple product is separately weak* continuous and the natural
embedding of $V$ into $V^{**}$ identifies $V$ as a subtriple of
$V^{**}$. Given a closed triple ideal $I \subset V$, the bidual $I^{**}$ embeds as a weak* closed
triple ideal in $V^{**}$, which can be decomposed into an $\ell^\infty$-sum $V^{**} = I^{**} \oplus_\infty J$ for some weak* closed triple ideal $J\subset V^{**}$,  orthogonal to $I^{**}$ \cite[lemma 3.3.16]{book}.

A linear map $\varphi: V \longrightarrow W$ between two JB*-triples is called
a {\it triple homomorphism} if $\{\varphi(a), \varphi(b),\varphi(c)\} =
\varphi\{a,b,c\}$ for $a,b, c \in V$. The triple isomorphisms between $V$ and $W$ are
exactly the surjective linear isometries (cf.\,\cite[Theorem 3.1.7, Theorem 3.1.20]{book}).

 A JB*-triple $V$ is called {\it abelian} if
its triple product satisfies
$$\{\{x,y,z\},u,v\} =\{x,\{y,z,u\},v\} =  \{x,y,\{z,u,v\}\}$$
for all $x,y,z,u, v \in V$.  An abelian C*-algebra is an abelian JB*-triple
and so is $C_\sigma(K)$ in the triple product defined above. In fact, $C_\sigma(K)$
is a closed subtriple of $C_0(K)$.

By \cite[Lemma 2.2, Theorem 3.7]{bc}, an
abelian closed subtriple $V$ of a C*-algebra admits a composition series
$(J_\lambda)_{0\leq \lambda\leq \mu}$ of closed triple ideals,  indexed by ordinals $\lambda$, such that
the quotient $J_{\lambda+1}/J_{\lambda}$ is linearly isometric to an
abelian C*-algebra, for  $\lambda< \mu$. We recall that  $(J_\lambda)_{0\leq\lambda\leq\mu}$
 is called a {\it composition series} if $J_0=\{0\}, J_\mu = V$ and for a limit ordinal $\lambda \leq \mu$, the ideal $J_\lambda$ is the closure
of $\bigcup_{\lambda' <\lambda} J_{\lambda'}$.

\section{Jordan structure in $C_\sigma$-spaces}\label{j}

We will make use of the abelian JB*-triple structure of $C_\sigma$-spaces to derive our result.
To pave the way, we first present some detailed analysis of this structure.

Let $V$ be an abelian closed subtriple of a C*-algebra (e.g. a $C_\sigma$-space)
in this section.  One
can consider it as  as a subtriple
 of  its bidual $V^{**}$ via the natural embedding
 $v\in V \mapsto \widehat v \in V^{**}$, where $\widehat v (\psi)
 = \psi (v)$ for $\psi \in V^*$. By \cite{bc, FR83}, $V^{**}$
is (isometric to and identified as) an {\it abelian} von Neumann
algebra with identity denoted by $\mathbf{1}$ and involution $z\in
V^{**} \mapsto z^* \in V^{**}$. The triple product in $V^{**}$ is
given by $\{a,b,c\} = ab^*c$. Each $\psi \in V^*$ can be viewed
naturally as a functional of $V^{**}$. If $\psi$ is a positive functional of
$V^{**}$, then $\psi (z^*) = \overline{\psi(z)}$ for each $z\in
V^{**}$. A positive functional $\psi\in V^*$ is called a {\it normal state} of $V^{**}$
if $\psi(\mathbf{1})=1$. It is called {\it pure} if it is an extreme point
of the norm closed convex set of normal states in $V^*$.

 Let $S$ be the set  of all pure normal
states of $V^{**}$, which are exactly the multiplicative normal
states of $V^{**}$. Given a projection $p\in V^{**}$ and $s\in S$,
we have $s(p)= 0$ or $1$ since $s(p)=s(p^2) = s(p)^2$. If $u\in
V^{**}$ is unitary, then $1=s(\mathbf{1})=s(u^*u) = |s(u)|^2$ for
all $s\in S$. We equip $S$ with the weak* topology of $V^*$ and
call it the {\it pure normal state space} of $V^{**}$.

The nonzero triple homomorphisms from $V$ to $\mathbb{C}$ are
exactly the set $K= {\rm ext}\, V^*_1$ of extreme points of the
dual unit ball $V^*_1$, where $K\cup \{0\}$ is weak*-compact
\cite[Proposition 2.3, Corollary 2.4]{FR83} and $S \subset K$.

For
each $\omega \in K$ and tripotent $c\in V^{**}$, we have
$$\omega(c)=\omega(cc^*c)=\omega(c)|\omega(c)|^2$$
which implies $\omega(c)=0$ or $\omega(c) \in \mathbb{T}$. We note
that $K$ is a $\mathbb{T}$-space with the natural
$\mathbb{T}$-action
$$\sigma: (\alpha,\omega) \in \mathbb{T}\times K \mapsto \alpha \omega \in
K$$ and we have $ K=\{\alpha s: \alpha \in \mathbb{T}, s\in S\}$.
In fact, each $\omega \in K$ has a {\it unique} representation
$\omega = \alpha s$ for some $\alpha \in \mathbb{T}$ and $s\in S$,
where $s= \overline{\omega(\mathbf{1})}\omega$. By \cite[Theorem
1]{FR83}, the map
\begin{equation}\label{id}
v\in V \mapsto \widehat v|_K \in C_\sigma(K)
\end{equation}
 is a surjective linear isometry, which enables us to identify $V$
 with the $C_\sigma$-space $C_\sigma (K)$.

\begin{remark} \label{cs} Let $\pi: C(\Omega) \longrightarrow C(\Omega)$ be a contractive projection.
Then its image $\pi(C(\Omega))$ is an abelian closed subtriple in some C*-algebra \cite{fr}
and hence the previous discussion implies
that it is a $C_\sigma$-space.
\end{remark}

For each $a\in V\backslash \{0\}$, let $V(a)$ be the JB*-subtriple generated by $a$ in $V$.
Then there is a surjective linear isometry and triple isomorphism
\begin{equation}\label{em}
\phi: C_0(S_a) \rightarrow V(a) \subset V
\end{equation}
which identifies $V(a)$
with the abelian JB*-triple
$C_0(S_a)$ of continuous functions vanishing at infinity on the
triple spectrum $S_a \subset (0, \|a\|]$, where $S_a \cup \{0\}$
is compact \cite[Theorem 3.1.12]{book}.

Let $\phi^{**}: C_0(S_a)^{**} \rightarrow V(a)^{**}$ be the
bidual map and let
$\frak{i}$ be the identity in the von Neumann algebra $C_0(S_a)^{**}$. Then
$e = \phi^{**}(\frak i)$ is the identity in the von Neumann algebra $V(a)^{**}$ with
 product and involution given by
 $$x \cdot y =
\{x,e,y\}, \quad  x \in V(a)^{**} \mapsto \{e,x,e\} \in V(a)^{**}.$$
While this abelian von Neumann algebraic
structure of $V(a)^{**}$ will be assumed throughout, it
should be noted that $V(a)^{**}$ need not be a subalgebra of
$V^{**}$ in its natural embedding. Nevertheless, $V(a)^{**}$ can
always be considered as a subtriple of $V^{**}$ and the identity
$e\in V(a)^{**}$ is a tripotent in $V^{**}$ satisfying
\begin{equation}\label{eae}
\{e, a,e\} = \{\phi^{**}(\frak i), \phi^{**}(\iota_a), \phi^{**}(\frak i)\} = \phi^{**}\{\frak i, \iota_a, \frak i\}
= a.
\end{equation}
For each $\rho\in K$, viewed as a complex-valued triple homomorphism on $V(a)^{**}$,
we have $\rho(e) =0$ if and only if $\rho(a)=\rho\{e,e,a\}=0$.

The norm closed triple ideal $J_a$ generated by $a$ in $V$ contains $V(a)$ and is the norm closure of $\{a, V,a\}$.
It has been shown in \cite[Lemma 2.2]{bc} that $J_a$ is  linearly isometric to the abelian C*-algebra $C_0(X_a)$ of continuous functions vanishing
at infinity  on
a locally compact Hausdorff space $X_a$. We need some detail here for later
application. In fact, $J_a$ is an abelian C*-algebra with the same product and involution of $V(a)^{**}$,
and $e\in V(a)^{**} \subset J_a^{**}$ is the identity of $J_a^{**}$, which can be seen from the following
computation using (\ref{eae}).
\[\{\{a,x,a\},e,\{a,y,a\}\} = \{a, \{x, \{a,e,a\}, y\},a\} \in \{a, V, a\};\]
\[ \{e, \{a,x,a\},e\} = ea^*xa^*e = ee^*ae^*xe^*ae^*e=ae^*xe^*a
 = \{x,a,a\} \in J_a. \]

Given $\rho \in K = {\rm ext}\,V^*_1$ with $\rho(e)=1$, its restriction $\rho|_{J_a} \in J_a^*$
 is a pure normal state of $J_a^{**}$. Conversely,
 each pure normal state $\vp$ of
$J_a^{**}$ is an extreme point of
the closed unit ball of $J_a^*$ and can be extended to an extreme
point $\widetilde \vp \in {\rm ext}\, V^*_1$ satisfying  $\widetilde \vp(e)=1$.
Let
\[X_a=\{\rho|_{J_a}: \rho \in K, \rho(e)=1\}\]
denote the pure normal state space of $J_a^{**}$ which is locally compact in the weak* topology $w(J_a^*,J_a)$ of $J_a^*$.

We note that for each $\rho \in K$, we have $\rho(a)=0$ if and only if $\rho(V(a))=\{0\}$, which in turn is equivalent
to $\rho(J_a)= \rho(\overline{\{a,V,a\}})=\{0\}$.

\begin{lem}\label{ja} In the above notation, the set $K(e)=\{\rho\in K: \rho(e)=1\}$ with the relative weak* topology
of $V^*$ is homeomorphic to $X_a=\{\rho|_{J_a}: \rho \in K, \rho(e)=1\}$ in the  topology  $w(J_a^*,J_a)$.
In particular, $K(e)$ is weak* locally compact in $K$.
\end{lem}

\begin{proof} We show that the restriction map $\rho \in K(e) \mapsto \rho|_{J_a} \in X_a$ is a homeomorphism in these topologies.
It is clearly continuous and surjective. Given $\rho, \rho'\in K$ such that $\rho|_{J_a} = \rho'|_{J_a}$, then
$\rho(a)=\rho'(a)\neq 0$ since $\rho(e)=\rho'(e)=1$. For any $v\in V$, we have $\{a,a,v\} \in J_a$ and hence $|\rho (a)|^2\rho(v) =
\rho\{a,a,v\}=\rho'\{a,a,v\} = |\rho'(a)|^2\rho'(v)$, giving $\rho(v)=\rho'(v)$. This proves injectivity of the map.

Finally, to show that the inverse of the map is continuous, let $(\rho_\gamma|_{J_a})$ be a net
converging to $\rho|_{J_a} \in X_a$. Then
again, for each $v\in V$, we have $\rho_\gamma (a) \rightarrow \rho(a) \neq 0$ and $|\rho_\gamma (a)|^2\rho_\gamma(v) =
\rho_\gamma \{a,a,v\} \rightarrow\rho\{a,a,v\} = |\rho(a)|^2\rho(v)$, which implies $\rho_\gamma(v) \rightarrow \rho(v)$,
proving continuity.
\end{proof}

\begin{remark}\label{xake}
The above lemma enables us to identify the pure normal state space $X_a$ with $K(e)$ and write
$X_a=\{\rho\in K: \rho(e)=1\}$.
\end{remark}

We retain the above notation in the sequel.

\section{Separably injective $C_\sigma$-spaces}

We characterize separably injective $C_\sigma$-spaces  in this section. Throughout, let $V$
be a $C_\sigma$-space.
We will identify $V$, as in the previous section, with
the $C_\sigma$-space $C_\sigma(K)$, where $K= {\rm ext}\, V^*_1$ is the set of nonzero triple homomorphisms
from $V$ to $\mathbb{C}$.

\begin{lem} \label{2}
Let $V$ be separably injective. Given $a\in V$ of unit norm and the identity
$e\in V(a)^{**}$, let $K(e) =\{\rho\in K: \rho(e)=1\}$.
Then there exists an element $v_a \in V$ such that $K(e) \subset K_a$
where
$$K_a = \{ \rho \in K: \rho (v_a)=1\}$$
and  $K_a$ is weak* compact in $V^*$.
\end{lem}
\begin{proof}
Since $K_a$
is weak* closed in $K \cup \{0\}$, it is weak* compact. Let
$$\phi: C_0(S_a) \rightarrow V(a) \subset V$$
be the embedding in (\ref{em}), where the triple spectrum $S_a\subset (0,1]$ can be identified, via
the evaluation map as usual, with the pure normal state space of $C_0(S_a)^{**}$.

Let $\chi_a$ be the constant function on $S_a$ with value $1$ 
 and consider the separable subspace
$C_0(S_a)+ \mathbb{C}\chi_a$ of $\ell^\infty (S_a)$. By separable injectivity of $V$,
the embedding $\phi: C_0(S_a)\rightarrow V(a) \subset V$
 admits a norm preserving extension
$\Phi :C_0(S_a) + \mathbb{C}\chi_a \rightarrow V$. Let $$v_a = \Phi(\chi_a)\in V.$$

To complete the proof, we show that $\rho(v_a) =1$ for each $\rho \in K(e)$.

We first observe that the sequence
$(r_n)$ of odd roots of the identity function $\iota_a$ in $C_0(S_a)$ converges
pointwise to the function $\chi_a$. Let $u_n = \phi(r_n)\in V(a)$.

Let $\rho \in K(e)$. Then the map $\rho \circ \phi : C_0(S_a) \rightarrow \mathbb{C}$
is a pure normal state of $C_0(S_a)^{**}$. Hence
we have $\rho (u_n) = \rho\circ \phi (r_n) \in [0,1]$ and
$\lim_n \rho(u_n)=1$.

The norm preserving extension $\Phi$ satisfies  $\|\Phi(\chi_a)\| \leq 1$ and
$$\|\Phi(\chi_a) - 2u_n\| = \|\Phi(\chi_a) - 2\phi(r_n)\|
=\|\Phi(\chi_a) - 2\Phi (r_n)\| \leq \|\chi_a- 2r_n\| \leq 1.$$
It follows that
 $|\rho(\Phi(\chi_a)) - 2\rho(u_n)| \leq 1$ for all $n$, which implies
$$|\rho(\Phi(\chi_a)) - 2|\leq 1. $$ Now $|\rho(\Phi(\chi_a))| \leq 1$ gives $\rho(v_a)=\rho(\Phi(\chi_a))=1$.
\end{proof}

Our next task is to show that a separably injective $C_\sigma$-space $V$  is actually linearly
isometric to an abelian C*-algebra. We adopt the following strategy. Since $V$ is abelian,
it has been noted in Section \ref{sect1} that  there is a composition series $(J_\lambda)_{0\leq \lambda \leq \mu}$
of closed triple ideals in $V$ such that for each ordinal $\lambda < \mu$, the
quotient $J_{\lambda +1}/J_\lambda$ is linearly isometric to
the C*-algebra $C_0(X_\lambda)$ of continuous functions vanishing at infinity on a locally compact
Hausdorff space $X_\lambda$, and $V^{**}$ is linearly isometric to the $\ell^\infty$-sum
$\bigoplus^{\ell^\infty}_{\lambda <\mu} (J_{\lambda +1}/J_\lambda)^{**} = \bigoplus^{\ell^\infty}_{\lambda <\mu} C_0(X_\lambda)^{**}.$
By the uniqueness of predual,
$V^*$ is linearly isometric to the $\ell^1$-sum
\[\bigoplus^{\ell^1}_{\lambda <\mu} (J_{\lambda +1}/J_\lambda)^* = \bigoplus^{\ell^1}_{\lambda <\mu} C_0(X_\lambda)^*.\]
Given that $V$ is separably injective, we will refine this construction
to show that $V$ is isometric to the abelian C*-algebra $\displaystyle\bigoplus^{c_0}_{\lambda< \mu} C_0( X_{\lambda})$.

Let $V$ be separably injective and let $a\in V$ be of unit norm. Consider the closed
subtriple $V(a)$ generated by $a$ in $V$ as well as the closed triple ideal $J_a$, the latter is linearly isometric
to the C*-algebra $C_0(X_a)$ as shown before, where $X_a=\{\rho\in K: \rho(e_a)=1\}$ is weak*
locally compact by Lemma \ref{ja} and Remark \ref{xake}, and $e_a$ is the identity of the von Neumann algebra $J_a^{**}$. By separable injectivity and Lemma \ref{2},
there exists $v_a\in V$ such that $\rho(v_a)=1$ for each $\rho \in X_a$.

 If $J_a= V$, then we are done. Otherwise we have the $\ell^\infty$-sum \[V^{**}=J_a^{**}\oplus_\infty
 (J_a^{**})^\bo\] where $(J_a^{**})^\bo$ is a nonzero weak* closed triple ideal in $V^{**}$, orthogonal to
 $J_a^{**}$, that is, $J_a^{**} \bo (J_a^{**})^\bo = \{0\}$ (cf.\,\cite[Lemma 3.3.16]{book}). The quotient map
$V^{**} \rightarrow V^{**}/J_a^{**}$ identifies $(J_a^{**})^\bo$ with the quotient $V^{**}/J_a^{**}$ and we can write
\begin{equation}\label{33}
V^{**} = J_a^{**} \oplus_\infty (V^{**}/J_a^{**}) = C_0(X_a)^{**}\oplus_\infty (V/J_a)^{**}.
\end{equation}

We have the $\ell^1$-sum  $V^{*} = C_0(X_a)^{*}\oplus_1 (V/J_a)^{*}$ where $J_a$ is an M-ideal in $V$.
 Consider the quotient map
\[x\in V \mapsto [x]:= x +J_a \in [V]:= V/J_a \]
which maps the closed unit ball of $V$ onto the closed unit ball of $V/J_a$ (cf.\,\cite[Corollary 5.6]{ae}).

Pick $[b] = b+J_a \in [V]$ with unit norm. Let $V([b])$ and $J_{[b]}$ be respectively the
closed subtriple and triple ideal generated by $[b]$ in $[V]$. We can repeat the previous arguments in the setting
$V([b]) \subset J_{[b]} \subset [V]$ to deduce that $J_{[b]}$ is linearly isometric to some abelian
C*-algebra $C_0(X_{[b]})$ with
\[X_{[b]} =\{\theta \in{\rm ext}\, [V]^*_1: \theta (e_{[b]})=1\}\]
where $X_{[b]}$ is locally compact in the weak* topology of $[V]^*$ by Lemma \ref{ja} and
$e_{[b]}$ is the identity of  $J_{[b]}^{**}\subset [V]^{**} =V^{**}/J_a^{**}$,  which identifies
with a tripotent $\widetilde e_b \in (J_a^{**})^\bo$ with $e_{[b]}= \widetilde e_b + J_{a}^{**}$.
Moreover, the quotient JB*-triple $[V]= V/J_a$ is abelian and seprably injective
by \cite[Proposition 4.6]{AC13}, and hence Lemma \ref{2} implies that
there exists $v_{[b]} =\widetilde v_b + J_a\in [V]= V/J_a$ such that $\theta (v_{[b]})=1$ for each
$\theta \in X_{[b]}$.

The set ${\rm ext}\, [V]^*_1$ consists of nonzero complex-valued triple homomorphisms on $[V]$,
 which can be lifted to nonzero complex triple homomorphisms on $V$ via the quotient map and we have
\[{\rm ext}\, [V]^*_1=\{\bar \rho: \rho \in K, \rho(J_a)=\{0\}\}\]
where $\bar \rho (x + J_a):= \rho(x)$.
Hence we have
\[X_{[b]}=\{ \bar\rho :
\rho \in K, \rho(a)=0, \rho(\widetilde e_b)=1\} \]
and $\bar \rho \in X_{[b]}$ implies $\rho(b) = \bar \rho([b]) \neq 0$ and  $\rho(\widetilde v_b)=1$.

A weak* convergent net in $[V]^*$ lifts to a weak* convergent net in $V^*$ via the quotient map.
Considering $X_{[b]}\subset C_0(X_{[b]})^*= J_{[b]}^*$ and  Lemma \ref{ja}, we see that
a net $(\bar\rho_\gamma)$ in $X_{[b]}$ converges to $\bar\rho\in X_{[b]}$
in the weak* topology of $C_0(X_{[b]})^*$ if and only if
 $(\rho_\gamma)$ weak* converges to $\rho$ in $K$. The homeomorphism \[\bar\rho \in X_{[b]}\mapsto \rho \in
 \{\rho \in K: \rho(a)=0, \rho(\widetilde e_b)=1\}\] enables us to identify these two spaces. We note that
\begin{equation}\label{10}
\mathbb{T}X_a \cap \mathbb{T}X_{[b]} \subset \mathbb{T}X_a \cap {\rm ext}\,[V]^*_1 = \emptyset
\end{equation}
where $\rho(a) \neq 0$ for all $\rho \in X_a$.

 In the $\ell^1$-sum
 \begin{equation}\label{11}
 V^* =J_a^* \oplus_1 (V/J_a)^*= C(X_a)^{*}\oplus_1 [V]^{*},
\end{equation}
  each $\omega \in V^*$ admits a decomposition
$\omega = \omega^1 + \omega^2$ in $V^*$ with $\omega ^2(J_a)=\{0\}$ and $\|\omega\|= \|\omega^1|_{J_a}\|+\|\omega^2\|$,
 which provides the identification of  $\omega$ as an element $(\widetilde \omega^1, \widetilde\omega^2)$ in the $\ell^1$-sum,
 defined by
\[\widetilde\omega^1 = \omega^1|_{J_a} \in J_a^* \quad {\rm and} \quad \widetilde\omega^2( [\cdot])
= \omega^2 (\cdot)
\in (V/J_a)^*.\]
Hence for an extreme point $\omega \in {\rm ext}\, V^*_1 =K$, we have $\widetilde\omega^1=0$
or $\widetilde\omega^2=0$.

Given a net $(\omega_\gamma)$ in $V^*$ weak* converging to a limit $\omega\in V^*$, it can be seen that
the net $(\widetilde\omega^1_\gamma)$ converges to $\widetilde \omega^1$ in the $w(J_a^*,J_a)$-topology
and the net $(\widetilde\omega^2_\gamma)$ converges to $\widetilde \omega^2$ in the
weak* topology of $(V/J_a)^*$. In particular, if the net $(\omega_\gamma)$ is in $K$
and $\omega \in K$ with $\widetilde\omega^j \neq 0$  $(j \in \{1,2\})$,
then the convergence of $(\widetilde\omega_\gamma^j)$ to $\widetilde\omega^j$ implies that $\widetilde
\omega_\gamma^j \neq 0$ eventually, and
hence  $\widetilde\omega_\gamma^{j'}=0$ for $j'\neq j$ eventually.

The closed unit ball of the $\ell^1$-sum in (\ref{11}) has extreme points
$$(\mathbb{T}X_a, 0) \cup (0, {\rm ext}\,[V]^*_1):= \{(\omega,0): \omega \in \mathbb{T}X_a\} \cup
\{(0,\omega): \omega \in {\rm ext}\,[V]^*_1 \}$$
and in the identification of (\ref{10}),   we have the disjoint union $K= \mathbb{T}X_a \cup {\rm ext}\,[V]^*_1$.
Given a net $(\omega_\gamma)$ in $K$ weak* converging to some $\omega \in K$,
and given either  $\omega\in \mathbb{T}X_a$ or $\omega \in {\rm ext}\,[V]^*_1$,  the above observation implies
 that  $\omega_\gamma$ belongs to the same set eventually.

Observe that $[J_b] = J_{[b]}=(J_b+J_a)/J_a$ and if $J_{[b]} \neq [V]$, we have the $\ell^\infty$-sum
\[ V^{**} = J_a^{**} \oplus_\infty  ((J_b+J_a)/J_a)^{**}\oplus_\infty ([V]/J_{[b]})^{**}
= C_0(X_a)^{**} \oplus_\infty  C_0(X_{[b]})^{**}\oplus_\infty ([V]/J_{[b]})^{**}\]
where the quotient JB*-triple $[[V]]:=[V]/J_{[b]}$ is separably injective, $\rho (v_a)=1$ for $\rho \in X_a$
and $\rho(\widetilde v_b)=1$ for $\rho \in X_{[b]}$.

The closed unit ball of the $\ell^1$-sum
\[
 V^* =  C_0(X_a)^* \oplus_1  C_0(X_{[b]})^{*}\oplus_1 ([V]/J_{[b]})^*
\]
has extreme points
\[(\mathbb{T}X_a,0,0) \cup (0,\mathbb{T}X_{[b]},0) \cup (0,0,  {\rm ext}\, [[V]]^*_1) \]
and we have the disjoint union $K= {\rm ext}\, V^*_1 = \mathbb{T}X_a \cup \mathbb{T}X_{[b]} \cup  {\rm ext}\, [[V]]^*_1$.
Given a net $(\omega_\gamma)$ in $K$ weak* converging to some $\omega \in K$,
if $\omega$ belongs to one  of the three sets above, then repeating the arguments as before,  $\omega_\gamma$ belongs to the same set eventually.

Now transfinite induction together with separable injectivity yields a composition series
$(J_\lambda)_{0\leq \lambda \leq \mu}$ of closed triple ideals in $V$, with $ v_\lambda \in V$, such that $V^{*}$ is linearly isometric to, and identifies with, the $\ell^1$-sum
\[\bigoplus^{\ell^1}_{\lambda<\mu} (J_{\lambda+1}/J_\lambda)^{*} = \bigoplus^{\ell^1}_{\lambda<\mu} C_0( X_{\lambda})^{*}\]
where $\rho( v_\lambda)=1$ for each $\rho \in  X_{\lambda}$, and the pure normal state space $X_\lambda$ of
$(J_{\lambda+1}/J_\lambda)^{**}$ identifies with the set
\[\{\rho \in K: \rho(J_\lambda)=\{0\}, \rho (\widetilde e_\lambda)=1\}\]
in which $\widetilde e_\lambda$ is the identity of $(J_{\lambda+1}/J_\lambda)^{**}$.
 In this identification, we have the disjoint union
\[ K= \bigcup_{\lambda<\mu} \mathbb{T}X_\lambda\]
and for a weak* convergent net $(\omega_\gamma)$ in $K$  with limit $\omega \in \mathbb{T}X_\lambda$ for
some $\lambda$, we have $(\omega_\gamma)$ in $\mathbb{T}X_\lambda$ eventually. As a consequence of Lemma \ref{ja}, in the identification $X_\lambda \subset C_0(X_\lambda)^*$ and
$X_\lambda \subset K$, the weak* convergence in $ C_0(X_\lambda)^*$ of a net $(\rho_\gamma)$  in $X_\lambda  $
to $\rho \in X_\lambda$ is the same as the weak* convergence in $K$.

\begin{lem}\label{x} Given that $V$ is separably injective and in the above notation, the subset $X_\lambda \cup\{0\}$
of  $K \cup\{0\}$ is weak* compact  for all $\lambda <\mu$ and also, $\mathbb{T}X_\lambda$ is relatively weak* open
in $K\cup \{0\}$.
\end{lem}
\begin{proof} Let $(\rho_\gamma)$ be a net in $X_\lambda$ weak* converging to a nonzero limit
$\omega \in V^*$. Then $\omega \in K$ and by the above remark, we must have $\omega \in\mathbb{T}X_\lambda$,
say $\omega = \alpha \rho$ with $\alpha \in \mathbb{T}$ and $\rho\in X_\lambda$. Since $X_\lambda$ is contained
in the weak* compact set $\{\rho'\in K: \rho'(v_\lambda)=1\}$, it follows that
$\alpha = \alpha \rho(v_\lambda) = \lim_\gamma \rho_\gamma(v_\lambda) =1$ and $\omega = \rho \in X_\lambda$.
This proves that $X_\lambda \cup \{0\}$ is weak* closed in $K\cup\{0\}$ and hence weak* compact.

For the second assertion, let $(\omega_\gamma)$ be a net in $(K\cup\{0\})\backslash \mathbb{T}X_\lambda$ weak* converging to some $\omega \in K$. Then again $\omega \notin \mathbb{T}X_\lambda$ for otherwise, the previous remark implies that
$\omega_\gamma$ belongs to $\mathbb{T}X_\lambda$
eventually which is impossible.
\end{proof}

The above construction enables us to show that a separably injective $C_\sigma$-space is isometric to
an abelian C*-algebra.

\begin{thrm}\label{4} Let $V$ be a separably injective $C_\sigma$-space.
Then $V$ is linearly isometric to an abelian C*-algebra.
\end{thrm}
\begin{proof}
Let  $K={\rm ext}\,V^*_1$
and as shown previously, we have the $\ell^1$-sum
\[ V^* =\bigoplus^{\ell^1}_{\lambda<\mu} C_0( X_{\lambda})^{*}\]
with the disjoint union $ K= \bigcup_{\lambda<\mu} \mathbb{T}X_\lambda$.
For each $\lambda <\mu$, there is an element $v_\lambda \in V$ such that $\rho(v_\lambda)=1$ for
all $\rho \in X_\lambda$.
We show that $V$ is linearly isometric
to the $c_0$-sum $\displaystyle\bigoplus^{c_0}_{\lambda< \mu} C_0( X_{\lambda})$ which would complete the proof.

We continue to identify $V$ with $C_\sigma (K)$ in (\ref{id}). By Lemma \ref{x}, each $f\in C_\sigma(K)$
 restricts to  a  continuous function $f|_{X_\lambda} \in C_0(X_\lambda)$.

We show that the map
$$f\in V \approx C_\sigma(K) \mapsto (f|_{X_\lambda}) \in
\bigoplus^{c_0}_{\lambda< \mu} C_0( X_{\lambda})$$ is a surjective linear isometry.

To see that $(f|_{X_\lambda})$ indeed belongs to the $c_0$-sum, we
need to show $(\|f|_{X_\lambda}\|) \in c_0(\Lambda)$, where $\Lambda=[0,\mu)$. Let
$\varepsilon >0$. By Lemma \ref{x},  for each $\lambda \in \Lambda$,
the set
$\mathbb{T}X_\lambda$ is relatively weak*
open in $K\cup\{0\}$.
Since $f$ vanishes at infinity on $K$, the set
$$U_\varepsilon = \{\omega \in K : |f(\omega)| < \varepsilon\}\cup \{0\}$$
is a relatively weak* open neighbourhood of $0$ in the one-point
compactification $K\cup\{0\}$ of $K$.  We have
$$K\cup\{0\} = \bigcup_{\lambda \in \Lambda} \mathbb{T}X_\lambda \cup
U_\varepsilon$$ and by weak* compactness, there are finitely many
$\lambda_1, \ldots, \lambda_n$ such that
$$K\cup\{0\} \subset \mathbb{T}X_{\lambda_1} \cup \cdots  \cup \mathbb{T}X_{\lambda_n} \cup
U_\varepsilon.$$ It follows that $\|f|_{X_\lambda}\| \leq
\varepsilon$ for $\lambda \notin \{\lambda_1, \ldots, \lambda_n\}$
which proves $(\|f|_{X_\lambda}\|) \in c_0(\Lambda)$.

 Since $K
=\displaystyle \cup_\lambda \mathbb{T}\,X_\lambda$ and each $f\in
C_\sigma (K)$ satisfies $f(\alpha \omega) = \alpha f(\omega)$ for $\alpha
\in \mathbb{T}$ and $\omega\in K$, it is evident that the map is a
linear isometry.

It remains to show that the map is surjective.
Let $(g_\lambda) \in \displaystyle\bigoplus^{c_0}_{\lambda \in \Lambda} C_0( X_{\lambda})$. Define a function $f: K \rightarrow \mathbb{C}$
by
$$f(\omega) = \alpha g_\lambda (\rho_\lambda) \quad {\rm for}\quad
\omega = \alpha \rho_\lambda \in \mathbb{T} X_\lambda.$$ The function
$f$ is well-defined since the sets
$\{\mathbb{T}X_\lambda\}_\lambda$ are mutually disjoint and each
$\omega\in K$ has a unique representation $\omega = \alpha \rho\in \mathbb{T}X_\lambda$ for
if $\alpha\rho =\beta \sigma \in \mathbb{T}X_\lambda $, we have $\alpha = \alpha\rho(v_\lambda)
=\beta\sigma(v_\lambda)= \beta$, where $X_\lambda$ is contained in the weak* compact
 set $\{\rho'\in K: \rho'(v_\lambda)=1\}$.

We complete the proof by showing $f\in C_\sigma (K)$. We have
readily  $f(\alpha \omega) = \alpha f(\omega)$ for $\alpha \in
\mathbb{T}$ and $\omega\in K$.

For continuity, let $(\omega_\gamma)$ be a net weak* converging to
$\omega \in K$ and say,   $\omega = \alpha \rho\in
\mathbb{T}X_{\lambda}$ for some $\lambda$.  By a previous remark, the net $(\omega_\gamma)$
is in $\mathbb{T}X_{\lambda}$ eventually. Therefore we have $\omega_\gamma
  = \alpha_\gamma \rho_\gamma$ with $\alpha_\gamma \in \mathbb{T}$ and $\rho_\gamma \in X_{\lambda}$  eventually.
  It follows that eventually $\alpha_\gamma = \alpha_\gamma \rho_\gamma (v_\lambda) \rightarrow
\alpha\rho(v_\lambda) = \alpha$  and $\rho_\gamma \rightarrow \rho$. Hence
 we have
\[\lim_\gamma f(\omega_\gamma)= \lim_\gamma \alpha_\gamma
g_\lambda(\rho_\gamma) = \alpha g_\lambda(\rho) = f(\omega).\]

Finally, for any $\varepsilon >0$, there are finitely many $\lambda_1, \ldots, \lambda_n$ such that
$\|g_\lambda\| < \varepsilon$ for $\lambda \notin \{\lambda_1, \ldots, \lambda_n\}$.
For each $\lambda_j$ with $j=1, \ldots,n$, there is a weak* compact set $E_j \subset X_{\lambda_j}$
such that $\{\rho \in X_{\lambda_j}: |g_{\lambda_j}(\rho)| \geq \varepsilon\} \subset E_j $.
This gives
\begin{eqnarray*}
\{\omega \in K: |f(\omega)| \geq \varepsilon\}= \cup_{j=1}^n \mathbb{T}\{
\rho\in X_{\lambda_j}:  |g_{\lambda_j}(\rho)| \geq \varepsilon\}
  \subset \cup_{j=1
}^n \mathbb{T}E_j \subset K
\end{eqnarray*}
where the finite union $ \cup_j \mathbb{T}E_j $
is weak* compact  and therefore $f \in C_0(K)$.
\end{proof}

Finally, by the characterisation of separably injective abelian C*-algebras in \cite[Theorem 3.5]{CL16},
together with Theorem \ref{4}, we conclude with the following main result of the paper.
\newpage
\begin{thrm}\label{111} Let $V$ be a $C_\sigma$-space.
The following conditions are equivalent.
\begin{itemize}
\item[(i)] $V$ is separably injective. \item[(ii)] $V$ is linearly
isometric to the Banach space $C_0(S)$ of complex continuous
functions vanishing at infinity on a substonean locally compact
Hausdorff space $S$.
\end{itemize}
\end{thrm}

\end{document}